\documentclass[11pt,letterpaper]{amsart}

\usepackage[dvipsnames,svgnames,x11names,hyperref]{xcolor}
\usepackage{bbm,url,graphicx,verbatim,amssymb,enumerate,booktabs,lmodern,mathtools,microtype,mathabx}
\usepackage[capitalize]{cleveref}
\usepackage{amsmath,amscd}
\usepackage[mathscr]{euscript}
\usepackage[all]{xy}
\usepackage[margin=1.2in]{geometry}
\usepackage{mathtools}
\usepackage{tikz}
\usetikzlibrary{arrows,positioning}

\newtheorem{theo}{Theorem}[section]
\newtheorem{lemm}[theo]{Lemma}
\newtheorem{prop}[theo]{Proposition}

\newtheorem*{theo*}{Theorem}

\theoremstyle{definition}
\newtheorem{defi}[theo]{Definition}
\newtheorem{cons}[theo]{Construction}

\newtheorem{nota}[theo]{Notation}
\newtheorem{exam}[theo]{Example}

\newtheorem{rem}[theo]{Remark}

\numberwithin{equation}{section}

\newcommand{\op}{^{\mathrm{op}}}
\newcommand{\cat}{\mathsf}
\newcommand{\on}{\operatorname}
\newcommand{\id}{\mathrm{id}}

\newcommand{\Hom}{\mathrm{Hom}}
\newcommand{\End}{\mathsf{End}}
\newcommand{\Aut}{\mathsf{Aut}}

\newcommand{\Fun}{\on{Fun}}
\newcommand{\map}{\on{map}}
\newcommand{\ehom}{\underline{\on{Hom}}}

\newcommand{\bring}{\cat{BRing}}
\newcommand{\cbring}{\cat{cBRing}}

\renewcommand{\ring}{\cat{Ring}}
\newcommand{\cring}{\cat{cRing}}

\newcommand{\bin}{\mathrm{Bin}}

\newcommand{\Z}{\mathbf{Z}}
\newcommand{\Q}{\mathbf{Q}}
\newcommand{\F}{\mathbf{F}}
\newcommand{\R}{\mathbb{R}}

\newcommand{\tG}{\tilde{\mathbb{G}}_a}

\newcommand{\sym}{\on{Sym}}

\title{Binomial rings and homotopy theory}
\author{Geoffroy Horel}

\begin{document}

\address{Université Sorbonne Paris Nord, Laboratoire Analyse, Géométrie et Applications, CNRS (UMR 7539), 93430, Villetaneuse, France.}
\email{horel@math.univ-paris13.fr}

\begin{abstract}
We produce a fully faithful functor from finite type nilpotent spaces to cosimplicial binomial rings, thus giving an algebraic model of integral homotopy types. As an application, we construct an integral version of the Grothendieck-Teichmüller group.
\end{abstract}



\maketitle

The goal of this paper is to develop an integral version of Sullivan's rational homotopy theory. Recall that Sullivan proved the following theorem.

\begin{theo*}[\cite{sullivaninfinitesimal}]
There is a functor $\Omega^*_{PL}$ from the homotopy category of simplicial sets to the opposite of the homotopy category of commutative differential graded algebras which is a left adjoint and whose right adjoint is denoted $A\mapsto \langle A\rangle$. When restricted to simplicial sets $X$ that are nilpotent and of finite type, the unit of this adjunction
\[X\to \langle \Omega^*_{PL}(X)\rangle\]
is a model for the localization with respect to rational homology isomorphisms.
\end{theo*}

This theorem implies in particular the following formula for the rational homotopy groups of a pointed space $X$ which is nilpotent and of finite type
\[\pi_i(X)\otimes\Q\cong \pi_i \langle \Omega^*_{PL}(X)\rangle.\]
In this formula, the $-\otimes\Q$ functor is the left adjoint to the inclusion of nilpotent groups in nilpotent uniquely divisible groups (in particular it is indeed given by simply tensoring with $\Q$ when restricted to abelian groups). The functor $\Omega^*_{PL}$ is a piecewise linear version of algebraic differential forms, it is quasi-isomorphic to the standard singular cochains functors.

It is very natural to try to extend Sullivan's theorem from rational homotopy theory to integral homotopy theory. There is no integral version of Sullivan's functor and in fact it can be shown that no strictly commutative model for integral cochains can exist as this would imply the triviality of Steenrod operations. Nevertheless, the singular cochain functor lands in $E_\infty$-algebras. There is in fact a more primitive invariant taking values in cosimplicial commutative rings given by taking degreewise the Hom from a simplicial set to the integers. 

Viewing the singular cochain functor as a functor to $E_\infty$-algebras in cochain complexes or cosimplicial commutative ring, there are precise analogues of Sullivan's theorem due to Mandell and To\"en respectively.

\begin{theo*}[Mandell \cite{mandellcochains}, To\"en \cite{toenprobleme}]
Let $X\mapsto C^*(X)$ (resp. $X\mapsto \Z^X$) be the singular cochain functor from the homotopy category of simplicial sets to the opposite of the homotopy category of $E_{\infty}$-differential graded algebras (resp. cosimplicial commutative rings). This functor is a left adjoint. The right adjoint is denoted $A\mapsto \langle A\rangle$. When restricted to simplicial sets $X$ that are nilpotent and of finite type, this functor is faithful. Moreover, two simplicial sets that are nilpotent and of finite type $X$ and $Y$ are weakly equivalent if and only if $C^*(X)$ and $C^*(Y)$ (resp. $\Z^X$ and $\Z^Y$) are weakly equivalent as $E_{\infty}$-differential graded algebras (resp. cosimplicial commutative rings).
\end{theo*}

The sticking point shared by these two theorems is the lack of fullness of the cochain functor. In the present paper, we suggest a way to fix this problem by attaching more structure to the functor $X\mapsto \Z^X$. We make the observation that the cosimplicial commutative ring $\Z^X$ is degreewise a binomial ring. Binomial rings are rings in which generalized binomial coefficients are well-defined. Alternatively, they are torsion-free lambda rings in which the identity map is a Frobenius lift at all prime. Our main result is then the following.

\begin{theo*}
There exists a model structure on cosimplicial binomial rings in which the weak equivalences are the quasi-isomorphisms. Let us denote by $c\mathcal{BR}ing$ the resulting $\infty$-category. The functor
\[X\mapsto \Z^X\]
from $\mathcal{S}$ to $c\mathcal{BR}ing\op$ is a left adjoint. The right adjoint is given by the formula
\[A\mapsto \map_{c\mathcal{BR}ing}(A,\Z).\]
The unit of this adjunction is equivalent to the map $X\to \Z_{\infty}(X)$ from $X$ to its Bousfield-Kan $\Z$-completion for $X$ connected of finite type. In particular, this map is a weak equivalence for $X$ nilpotent of finite type.
\end{theo*}

Following To\"en \cite{toenchamps}, we also make a version of a binomial affine homotopy type attached to a homotopy type. Given a homotopy type $X$, we may construct the functor $X^{bin}$ from the category of binomial rings to the $\infty$-category of spaces given by the formula
\[X^{bin}(R)=\map_{c\mathcal{BR}ing}(\Z^{X},R).\]
We show that the functor $X\mapsto X^{bin}$ is fully faithful when restricted to nilpotent and finite type spaces. Moreover this object recovers the various localizations of $X$. For any subring $R$ of $\Q$, the space $X^{bin}(R)$ is the $R$-localization of $X$ while the space $X^{bin}(\Z_p)$ is the $p$-completion of $X$. This object can be used to define a ``binomial algebraic group'' of homotopy automorphisms of a space. As an application of this theory we construct an integral version of the Grothendieck-Teichmüller group. This is a functor from binomial rings to groups that includes both the pro-algebraic Grothendieck-Teichmüller group and the pro-$p$ Grothendieck-Teichmüller group. 

\subsection*{Related work}

To the best of the author's knowledge, the first occurence of binomial rings in homotopy theory is in \cite{wilkersonapplications}. In this paper, Wilkerson constructs functor from binomial rings to homotopy types which is very likely to be equivalent to our functor $X\mapsto X^{bin}$ constructed in Section \ref{section : binomial homotopy type}. This paper was also very much influenced by the paper \cite{ekedahlminimal} by Ekedahl, in which a model for homotopy types is constructed in the category of simplicial free abelian groups and numerical maps. Although we do not know exactly how to compare the present work to Ekedahl's work we believe that there is a strong relationship. In particular, our Proposition \ref{prop : case n=1} seems to be very similar to the proof of \cite[Lemma 2.7]{ekedahlminimal}. This paper also owes a lot to the papers \cite{toenchamps} and \cite{toenprobleme} that are themselves influenced by Grothendieck's \emph{Pursuing stacks}. In particular, our Proposition \ref{prop : main theorem for EM spaces} is essentially equivalent to \cite[Corollaire 3.4]{toenprobleme}.

In a different direction, the paper \cite{yuanintegral} by Allen Yuan gives a different model of integral homotopy types using cochains valued in the sphere spectrum. Even though there does not seem to be a direct comparison functor between Yuan's model and ours, both rely on the idea of homotopy trivializing the Frobenius map. Finally let us mention the work \cite{blomquistintegral} in which the authors identify simply connected spaces with a full subcategory of coalgebras over a certain comonad on chain complexes. In the current work we identify finite type simply connected spaces with a full subcategory of algebras over a certain monad on cochain complexes (Theorem \ref{theo : monadic}). The additional finite type assumption comes from the fact that we are working cohomologically instead of homologically.

\subsection*{Acknowledgment}

I wish to thank Alexander Berglund and Dan Petersen for helpful conversations.

This work was partially written at Institut Mittag-Leffler in Sweden during the semester \emph{Higher algebraic structures in algebra, topology and geometry}.

\subsection*{Conventions}

We denote by $\mathbb{N}$ the set of non-negative integers. 

We denote ordinary categories in sans-serif characters and the underlying $\infty$-categories of a model category in calligraphic characters. For instance, we denote by $\cbring$ the model category of cosimplicial commutative rings and by $c\mathcal{BR}ing$ the underlying $\infty$-category. We denote by $\mathcal{S}$ the $\infty$-category of spaces.

If $\cat{C}$ is a simplicial category, we denote the mapping simplicial set by $\map_{\cat{C}}$. If $\mathcal{C}$ is an $\infty$-category, we denote by $\map_{\mathcal{C}}$ the infinity-categorical mapping space. 

\section{Binomial rings}

We use the word ring for commutative unital rings. We denote by $\ring$ the category of rings. 

\begin{defi}
A \emph{binomial ring} is a ring $R$ whose underlying abelian group is torsion-free and such that, for all $a\in R$ and $n\in\mathbb{N}$, the product
\[\prod_{i=0}^{n-1}(a-i)\]
is divisible by $n!$.
\end{defi}

\begin{nota}
In the following, for $a$ an element of a binomial ring and $n\in\mathbb{N}$, we write
\[\binom{a}{n}=\frac{\prod_{i=0}^{n-1}(a-i)}{n!}.\]
\end{nota}

\begin{rem}
Binomial rings appear under the name numerical ring and with a different axiomatic in \cite{ekedahlminimal}. The equivalence between the two notions is proved in \cite{xantchabinomial}. We can also alternatively define a binomial ring as torsion free commutative ring in which the identity map is a Frobenius lift for any prime \cite[Theorem 4.1]{elliottbinomial}. As such binomial rings are particular lambda rings. 
\end{rem}

We shall recall in this section some important facts about the category of binomial rings. All of the results can be found in \cite{elliottbinomial}. We denote by $\bring$ the category of binomial rings. Morphisms of binomial rings are, by definition, morphisms of rings. So there is a forgetful functor
\[U:\bring\to\ring\]
which is fully faithful. This functor has a left adjoint denoted $\bin^U$. Explicitly, for a ring $A$, the ring $\bin^U(A)$ is the intersection of all binomial subrings of $A\otimes_{\Z}{\Q}$ containing the image of $A\to A\otimes_{\Z}\Q$. 

The forgetful functor $U$ also has a right adjoint denoted $\bin_U$. Given a ring $A$ the ring $\bin_U(A)$ is the subring of the ring of big Witt vectors $W(A)$ consisting of elements that are fixed by all the Frobeniuses.

The existence of these adjoints implies that the forgetful functor $U$ preserves limits and colimits. In particular, the conditions of Beck's monadicity theorem (see \cite[Chapter 3, Theorem 3.14]{barrwells}) are satisfied and the adjunction
\[\bin^U:\ring\leftrightarrows \bring:U\]
exhibits $\bring$ as the category of algebras over the associated monad.

Given a set $S$, the ring $\bin^U(\Z[S])$ is the ring of numerical polynomials in $S$ variables. This ring is simply the subring of $\Q[S]$ whose elements are the polynomials $f$ in $S$-variables with rational coefficients such that $f(\Z^S)\subset \Z$. Note that this ring is strictly larger than $\Z[S]$. For example, the polynomial
\[F_p(X)=\frac{X^p-X}{p}\]
with $p$ a prime number is a polynomial whose coefficients are not integers but which is such that $F_p(n)\in\Z$ for any $n\in\Z$.

\begin{nota}
We denote by $\on{Num}[S]$ the ring $\bin^U(\Z[S])$ of numerical polynomials in $S$ variables. When the variables are called $x_1,\ldots,x_n$, we simply write $\on{Num}[x_1,\ldots,x_n]$.
\end{nota}

\section{The cobar construction of $\on{Num}[x]$}

In this section, we compute the cohomology of the cobar construction of the ring $\on{Num}[x]$. This computation will play an important role in the proof of our main theorem.

Let $C$ be a coalgebra over $\Z$ which is flat as an abelian group. Assume that we are further given a coaugmentation of $C$, i.e. a map of coalgebras $\Z\to C$. Then the cobar  construction of $C$ is the cochain complex $\Omega(C)$ associated to the cosimplicial object
\[[n]\mapsto C^{\otimes n}\]
in which the inner cofaces are obtained by applying the diagonal map and the outer cofaces are given by the coaugmentation. 

\begin{rem}
A pointed space can be viewed as a coaugmented coalgebra in the category of spaces (the coalgebra structure is given by the diagonal map and the coaugmentation by the base point). In that context, the totalization of the analogous cosimplicial object is a model for the loop space. This explains the notation for the cobar construction.
\end{rem}

The goal of this section is to compute the cobar construction of $\on{Num}[x]$. On top of being a commutative ring, $\on{Num}[x]$ is a cocommutative coalgebra with coproduct 
\[\Delta:\on{Num}[x]\to \on{Num}[x]\otimes \on{Num}[x]\cong \on{Num}[x,y]\]
given by the formula
\[\Delta(f)(x,y)=f(x+y).\]

\begin{prop}\label{prop : cobar of Num}
The cohomology of $\Omega(\on{Num}[x])$ is free of rank $1$ in cohomological degree $0$ and $1$ and is zero in any other degree.
\end{prop}

\begin{proof}
The coalgebra $\on{Num}[x]$ is, as an abelian group, isomorphic to the direct sum $\oplus_{n\in\mathbb{N}}\binom{x}{n}$ by \cite[Lemma 2.2]{elliottbinomial}. We have the Vandermonde identity which describes the diagonal in $N$~:
\[\Delta\binom{x}{n}=\binom{x+y}{n}=\sum_{p+q=n}\binom{x}{p}\binom{y}{q}.\]
From this formula, we see that the coalgebra $N$ is graded by declaring $\binom{x}{n}$ to be of degree $n$ (beware that, with this convention, the two maps $\Z[x]\to\on{Num}[x]$ and $\on{Num}[x]\to\Q[x]$ are not maps of graded abelian groups).

It follows that the cohomology of the cobar construction of $\on{Num}[x]$ inherits an extra grading on top of the cohomological grading. We shall show that it is free of rank $1$ in bidegrees $(0,0)$ and $(1,1)$ and is zero otherwise (our convention is cohomological degree then internal degree). Since the cochain complex $\Omega(\on{Num}[x])$ is finitely generated and free in each degree, we may apply the duality functor $(-)^{\vee}=\Hom(-,\Z)$ to it and obtain a chain complex of graded abelian groups. We shall prove that the homology of $\Omega(\on{Num}[x])^{\vee}$ is free of rank $1$ in (homological) bidegrees $(0,0)$ and $(1,1)$. By the universal coefficient theorem, this will imply the desired result.

In order to prove this final claim, it suffices to make two observations. First the dual of $\on{Num}[x]$ (in the graded abelian group category) is the graded algebra $\Z[x]$ thanks to the explicit formula that we have for the diagonal. Second, the chain complex $\Omega(\on{Num}[x])^{\vee}$ can be identified immediately with the bar construction of $\on{Num}[x]^{\vee}\cong\Z[x]$. The homology of the bar construction of $\Z[x]$ is well-known to be of the desired form.
\end{proof}

\section{Cosimplicial binomial rings}

We denote by $\cat{cAb}$ the category of cosimplicial abelian groups. Recall that the Dold-Kan equivalence gives an equivalence of categories between $\cat{cAb}$ and $\cat{Ch}^*(\Z)$ the category of non-negatively graded cochain complexes (or non-positively graded chain complexes). This equivalence is realized by the normalized cochain complex functor
\[N:\cat{cAb}\to\cat{Ch}^*(\Z)\]
that sends a cosimplicial abelian group $A^{\bullet}$ to the cochain complex given in degree $n$ by the intersection of the kernels of all the codegeneracy maps.
\[N(A)^n:=\bigcap_{i=0}^n\mathrm{ker}s^i.\]
The differential on $N(A)$ is given by the alternating sum of coface maps. The inverse of this functor is denoted $\Gamma$. In order to simplify notations, we will often omit $\Gamma$ and allow ourselves to view a cochain complex as a cosimplicial abelian group. We shall also use homological grading. So we write $\Z[-n]$ for the cosimplicial abelian group obtained by applying $\Gamma$ to the chain complex given by $\Z$ in cohomological degree $n$ and zero in any other degree.

We denote by $\cring$ the category of cosimplicial rings.  Similarly, we denote by $\cbring$ the category of cosimplicial binomial rings. We extend the functors $U$, $\bin^U$ ann $\bin_U$ to $\cring$ and $\cbring$ by applying them in each cosimplicial degree. We obtain adjunctions
\[\bin^U:\cring\leftrightarrows \cbring:U\]
and
\[U:\cbring\leftrightarrows \cring:\bin_U.\]

\begin{theo}
There is a simplicial model structure on $\cat{cAb}$ in which the weak equivalences are the quasi-isomorphism and the fibrations are the epimorphisms.
\end{theo}

\begin{proof}
This is simply the projective model structure on cochain complexes of abelian groups transferred along the Dold-Kan equivalence
\[N:\cat{cAb}\leftrightarrows \cat{Ch}^*(\Z):\Gamma\]
It suffices to check that the fibrations are exactly the epimorphisms. This is the case in the projective model structure of cochain complexes and epimorphisms are preserved and reflected along an equivalence of categories. 

This model structure is simplicial. The cotensoring is given by the following formula
\[A^X:=[p]\mapsto (A^p)^{X_p}:=\prod_{X_p}A^p.\]
where $A$ is a cosimplicial abelian group and $X$ is a simplicial set.
\end{proof}

\begin{rem}
In cochain complexes or cosimplicial abelian groups, the epimorphisms are exactly the degreewise epimorphisms.
\end{rem}

We shall now construct a simplicial model structure on the category $\cring$. The cotensoring of an object of $\cring$ by a simplicial set is the same as in $\cat{cAb}$ while the tensor of $A^\bullet$ by $K\in\cat{sSet}$ is given by the following coend
\[( K_{\bullet}\otimes A^{\bullet})^n=(K\times\Delta^n)_\bullet\otimes_{\Delta}A^{\bullet}.\]
This formula can be found in \cite[Section 2.10]{bousfieldcosimplicial}.

\begin{theo}
There is a simplicial model structure on $\cring$ in which 
\begin{enumerate}
\item The weak equivalences are the quasi-isomorphisms.
\item The fibrations are the maps that are degreewise epimorphisms of abelian groups.
\end{enumerate}
\end{theo}

\begin{proof}
This is a standard transfer theorem. This can be found in \cite[Théorème 2.1.2]{toenchamps}. The model structure on $\cat{cAb}$ is cofibrantly generated and all objects are fibrant. Moreover, there is a functorial path object in $\cring$ given by $A\mapsto A^{I}$ where we denote by $I$ the simplicial set represented by the object $[1]$ in $\Delta$. Clearly, in the factorization
\[A\to A^I\to A\times A\]
the first map is a weak equivalence and the second map is a fibration because the same is true in $\cat{cAb}$. Thus the model structure on $\cat{cAb}$ can be transferred thanks to the path object argument \cite[Lemma 2.3]{schwedealgebras}.
\end{proof}

\begin{rem}
It is not the case that an epimorphism of commutative rings is necessarily an epimorphism of the underlying abelian groups. A counter-example is given by the unit map $\Z\to\Q$.
\end{rem}

We shall now show that the forgetful functor preserves geometric realization. This will rely on the following proposition.

\begin{prop}\label{prop: forgetful preserves geometric realization}
Let $M$ be a simplicial category, let $T$ be a simplicial monad on $M$ and 
\[F_T:M\leftrightarrows M^T:U_T\]
be the corresponding adjunction. Assume that $T:M\to M$ preserves geometric realization of simplicial diagrams. Then $U_T$ preserves geometric realization of simplicial diagrams.
\end{prop}

\begin{proof}
Let $A_\bullet$ be a simplicial diagram in $M^T$. Then, by hypothesis, we have an isomorphism
\[T|U_T(A_\bullet)|\cong |T(U_T (A_\bullet))|\] 
This makes the object $|U_T(A_\bullet)|$ into a $T$-algebra via the map
\[T|U_T(A_\bullet)|\cong |T(U_T (A_\bullet))|\to |U_T(A_\bullet)|\]
where the second map uses the $T$-structure of each of the objects $A_n$. We shall show that the object $|U_T(A_\bullet)|$ equipped with this $T$-algebra structure satisfies the universal property for being the geometric realization (in $M^T$) of the simplicial diagram $A_\bullet$. This will imply that the canonical map
\[|U_T(A_\bullet)|\to U_T|A_\bullet|\]
is an isomorphism as desired.

Now, we prove the claim. Let $B$ be an object of $M^T$, we have an equalizer diagram
\[\map_{M^T}(|U_T(A_\bullet)|,B)\to \map_{M}(|U_T(A_\bullet)|,U_TB)\rightrightarrows \map_{M}(T|U_T(A_\bullet)|,U_TB)\]
By our hypothesis, and the universal property of the geometric realization (a particular type of weighted colimit), we can rewrite this equalizer diagram as
\[\map_{M^T}(|U_T(A_\bullet)|,B)\to \mathrm{Tot}\map_{M}(U_T(A_\bullet),U_TB)\rightrightarrows \mathrm{Tot}\map_{M}(TU_T(A_\bullet),U_TB)\]
where $\mathrm{Tot}$ denotes the totalization in simplicial sets. Since $\mathrm{Tot}$ commutes with equalizers, we have
\begin{align*}
\map_{M^T}(|U_T(A_\bullet)|,B)&\cong \mathrm{Tot}[\mathrm{eq}(\map_{M}(U_T(A_\bullet),U_TB)\rightrightarrows \map_{M}(TU_T(A_\bullet),U_TB))]\\
&\cong \mathrm{Tot}[\map_{M^T}(A_\bullet,B)]\\
&\cong \map_{M^T}(|A_\bullet|,B)
\end{align*}
\end{proof}

As a corollary, we obtain the following proposition.

\begin{prop}\label{prop : forget preserves geometric realization}
The forgetful functor
\[\cring\to \cat{cAb}\]
preserves geometric realizations of simplicial diagrams.
\end{prop}

\begin{proof}
Thanks to the previous proposition, it suffices to prove that the free commutative algebra monad preserves geometric realization. This monad is made using direct sums, orbits and the functors $X\mapsto X^{\otimes n}, n\in\mathbb{N}$. These three operations do preserve geometric realization.
\end{proof}

There is a similar model structure on $\cbring$.

\begin{theo}\label{theo : model structure on binomial rings}
There is a model structure on $\cbring$ in which
\begin{enumerate}
\item The weak equivalences are the quasi-isomorphisms.
\item The fibrations are the maps that are degreewise epimorphisms of abelian groups.
\end{enumerate}
Moreover, this model structure is simplicial and the adjunction
\[\bin^U:\cring\leftrightarrows \cbring:U\]
is a Quillen adjunction.
\end{theo}

\begin{proof}
The proof is the same, using the fact that $A\mapsto A^I$ can also serve as a path object in $\cbring$.
\end{proof}

\begin{prop}\label{prop : U preserves tensoring}
The forgetful functor
\[U:\cbring\to\cring\]
preserves the tensoring by simplicial sets.
\end{prop}

\begin{proof}
In both categories, the tensor of $A^\bullet$ by $K\in\cat{sSet}$ is given by the following coend
\[( K_{\bullet}\otimes A^{\bullet})^n=(K\times\Delta^n)_\bullet\otimes_{\Delta}A^{\bullet}.\]
The forgetful functor has a right adjoint so it preserves all colimits and tensoring by sets (as those are simply given by coproducts) so it will preserve the tensoring by simplicial sets.
\end{proof}

\begin{prop}\label{prop: W preserves hocolim}
The forgetful functor
\[W:\cring \to\cat{cAb}\]
preserves and reflects filtered homotopy colimits and homotopy colimits indexed by $\Delta\op$.
\end{prop}

\begin{proof}
Since this functor is homotopy conservative, it is enough to show that it preserves those homotopy colimits. Filtered colimits in cosimplicial abelian groups are automatically derived and the forgetful functor does preserve ordinary filtered colimits so the first claim is obvious. For homotopy colimits indexed by $\Delta\op$, we use Proposition \ref{prop: forgetful preserves geometric realization} which proves that the forgetful functor preserves geometric realizations. In a model category, geometric realization coincides with homotopy colimits for Reedy cofibrant objects. It turns out that, in cosimplicial abelian groups, geometric realizations are automatically derived since every simplicial diagram of cosimplicial abelian groups is Reedy cofibrant.
\end{proof}

\begin{prop}\label{prop : U preserves hocolim}
The forgetful functor
\[U:\cbring\to\cring\]
preserves and reflects homotopy colimits.
\end{prop}

\begin{proof}
Since this functor is homotopy conservative, it suffices to prove that it preserves homotopy colimits. For this, it is enough to prove that it preserves finite homotopy coproducts, filtered homotopy colimits and homotopy colimits indexed by $\Delta\op$. The case of filtered colimits is straightforward since these are computed in $\cat{cAb}$. The case of finite coproducts is also easy as these are simply given by the derived tensor product in both categories and in $\cbring$, tensor products are automatically derived since a binomial ring is flat.

For homotopy colimits indexed by $\Delta\op$, using the previous lemma, it is enough to  prove that the forgetful functor
\[V:\cbring\to \cat{cAb}\]
preserves them. Moreover, arguing as in the proof of the previous lemma, it is enough to show that $V$ preserves geometric realizations of simplicial diagrams. We already know that the forgetful functor $\cring\to\cat{cAb}$ preserves geometric realizations. The forgetful functor $U:\cbring\to\cring$ also preserves geometric realizations. Indeed, this functor preserves colimits and the tensoring by simplicial sets thanks to Proposition \ref{prop : U preserves tensoring}. So $V$ preserves geometric realizations as it is the composite of two functors that do so.
\end{proof}

\begin{rem}
Even though the functor $U:\cbring\to\cring$ is fully faithful, the induced functor between the homotopy categories is not. Indeed, it follows from our main theorem (Theorem \ref{theo : main}) that the mapping spaces between cosimplicial binomial rings of the form $\Z^X$ for a reasonable space are weakly equivalent to the mapping spaces between the corresponding spaces. However, this is not the case in $\cring$. This kind of situation happens quite often in homotopy theory (for example, the $1$-category of chain complexes of $\F_p$-vector spaces is a full subcategory of the category of chain complexes of abelian groups but this is no longer true after inverting quasi-isomorphisms). Nevertheless, this forgetful functor is monadic as shown by the following theorem.
\end{rem}

\begin{theo}\label{theo : monadic}
Let us denote by $c\mathcal{BR}ing$ and $c\mathcal{R}ing$ the $\infty$-categories underlying the model categories $\cbring$ and $\cring$. The right adjoint functor 
\[U:c\mathcal{BR}ing\to c\mathcal{R}ing\]
exhibits the $\infty$-category $c\mathcal{BR}ring$ as the $\infty$-category of algebras over the corresponding monad. The same can be said for the right adjoint functor
\[V:c\mathcal{BR}ing\to \mathcal{D}(\Z)_{\leq 0}\]
\end{theo}

\begin{proof}
This is an application of the Barr-Beck-Lurie monadicity theorem. In both cases, the functor is conservative and preseves colimits indexed by $\Delta\op$ thanks to Proposition \ref{prop: W preserves hocolim} and Proposition \ref{prop : U preserves hocolim}.
\end{proof}

\section{Main theorem for Eilenberg-MacLane spaces}

We define a functor $\cat{sSet}\to \cbring\op$ by the formula $X\mapsto \Z^X$. Since $\Z^{X}$ is degreewise a product of copies of $\Z$, it is indeed a cosimplicial binomial ring. This functor is a left adjoint. Its derived right adjoint is denoted
\[A\mapsto \langle A\rangle.\] 
Explicitly, this derived right adjoint is given by
\[A\mapsto \map_{\cbring}(A^c,\Z)\]
where $A^c\xrightarrow{\simeq}A$ is a cofibrant replacement.
Our main theorem, Theorem \ref{theo : main} below, is that the derived unit map 
\[X\to\langle \Z^X\rangle\]
is a weak equivalence for reasonable spaces. In this section we shall focus on the case of Eilenberg-MacLane spaces.

We denote by $K_n$ an Eilenberg-MacLane space of type $(\Z,n)$. An explicit model is given by the inverse simplicial Dold-Kan functor applied to the chain complex $\Z[n]$. The counit of the free-forgetful adjunction between simplicial abelian groups and simplicial sets gives us a map of simplicial abelian groups
\[\Z\langle K_n\rangle \to \Z[n]\]
(where $\Z\langle-\rangle$ is our notation for the free abelian group functor). If we apply the duality functor $M\mapsto \Hom(M,\Z)$ to this map, we obtain a map  of cosimplicial abelian groups
\[\Z[-n]\to \Z^{K_n}.\]
Finally, we can use the free-forgetful adjunction between cosimplicial abelian groups and cosimplicial binomial rings to produce a map of cosimplicial binomial rings
\[\alpha_n:\sym^{bin}(\Z[-n])\to \Z^{K_n},\]
where $\sym^{bin}(A):=\bin^U(\sym(A))$ is our notation for the free binomial ring on an abelian group $A$.

\begin{prop}\label{prop : case n=1}
The map $\alpha_1$ is a weak equivalence.
\end{prop}

\begin{proof}
This fact and the next Theorem are equivalent to \cite[Corollaire 3.4]{toenprobleme}. We shall give an independant proof. We first compute $\sym^{bin}(\Z[-1])$. The simplicial abelian group $\Z[1]$ is the usual simplicial bar construction for $\Z$ given by
\[[n]\mapsto \Z^{n}\]
with face maps given by addition of two consecutive factors or the zero map $\Z\to 0$ for the two extreme face maps. The cosimplicial abelian group $\Z[-1]$ is its dual. It is given by
\[[n]\mapsto \Z^{n}\]
with coface maps induced by the diagonal map $\Z\to\Z\times \Z$ or the zero map $0\to\Z$ for the two extreme coface maps. Therefore the cosimplicial binomial ring $\sym^{bin}(\Z[-1])$ is given by
\[[n]\mapsto \on{Num}[x]^{\otimes n}\]
with cofaces induced by the comultiplication map
\[\on{Num}[x]\to\on{Num}[x]\otimes\on{Num}[x]\]
or the unit $\Z\to\on{Num}[x]$ for the two extreme coface maps. This cosimplicial object is simply the cobar construction of the coalgebra $\on{Num}[x]$. Thanks to Proposition \ref{prop : cobar of Num}, we know that the cohomology of this object is abstractly isomorphic to the cohomology of $\Z^{K_1}\simeq\Z^{S^1}$. 

It remains to show that the map $\alpha_1$ induces this isomorphism. For this we have to analyze the map of cosimplicial commutative rings
\[\alpha_1:\sym^{bin}(\Z[-1])\to\Z^{K_1}\]
The cosimplicial ring $\Z^{K_1}$ is given by
\[[n]\mapsto \Hom_{\cat{Set}}(\Z^n,\Z)\]
and the cosimplicial ring $\sym^{bin}(\Z[-1])$ is the cosimplicial subring obtained by restricting to numerical maps in each degree. There is an even smaller cosimplicial abelian group given by taking abelian group homomorphisms in each cosimplicial degree. The latter cosimplicial abelian group is simply $\Z[-1]$. The map from this cosimplicial abelian group to $\Z^{K_1}$ induces an isomorphism in cohomology in degree $1$. It follows that the map
\[\alpha_1:\sym^{bin}(\Z[-1])\to\Z^{K_1}\]
must be surjective on cohomology in degree $1$. Since both the source and target cohomology group are free of rank $1$, the map $\alpha_1$ is an isomorphism in this degree.

The fact that $\alpha_1$ is an isomorphism in degree $0$ is completely straightforward.
\end{proof}

\begin{prop}\label{prop : main for EM}
The map $\alpha_n$ is a weak equivalence of cosimplicial binomial rings.
\end{prop}

\begin{proof}
In order to simplify notations, we denote by $F_n$ the cosimplicial binomial ring $\sym^{bin}(\Z[-n])$.The case $n=1$ is treated in the previous proposition. Assume that the result is known for $n$. We have a homotopy pushout square that we call $S_1$
\[
\xymatrix{
F_{n+1}\ar[r]\ar[d]&\Z\ar[d]\\
\Z\ar[r]&F_{n}
}
\]
obtained by applying the left Quillen functor $\sym^{bin}$ to the following homotopy pushout square in $\cat{cAb}$~:
\[
\xymatrix{
\Z[-(n+1)]\ar[r]\ar[d]&0\ar[d]\\
0\ar[r]&\Z[-n]
}
\]
Similarly, we have a homotopy pullback square
\[\xymatrix{
K_n\ar[r]\ar[d]&\ast\ar[d]\\
\ast\ar[r]&K_{n+1}
}
\]
which induces another homotopy pushout square in $\cbring$ that we call $S_2$. 
\[\xymatrix{
\Z^{K_{n+1}}\ar[r]\ar[d]&\Z\ar[d]\\
\Z\ar[r]&\Z^{K_n}
}
\]
The fact that $S_2$ is a homotopy pushout square is a consequence of the analogous fact in the model category $\cring$ (see \cite[Theorem A.1]{toenprobleme}) and Proposition \ref{prop : U preserves hocolim}.

The maps $\alpha_n$ and $\alpha_{n+1}$ induce a map from $S_1$ to $S_2$. Thanks to Proposition \ref{prop : U preserves hocolim}, homotopy pushouts in $\cbring$ coincide with homotopy pushouts in $\cring$ and are simply given by the relative derived tensor product. In this particular situation, the relative derived tensor product $\Z\otimes^{L}_A\Z$ is the Bar construction of $A$ usually denoted $B(A)$. 

In other words, we would like to prove that the map $\alpha_{n+1}$ is a weak equivalence knowing that $\alpha_n=B(\alpha_{n+1})$ is a weak equivalence. This follows from Theorem \ref{conservativity of Bar}.
\end{proof}

\begin{prop}\label{prop : main theorem for EM spaces} 
The unit map $\eta_n:K_n\to\langle \Z^{K_n}\rangle$ is a weak equivalence.
\end{prop}

\begin{proof}
Let us consider the composite
\[\beta_n:K_n\to\langle \Z^{K_n}\rangle\to \langle F_n\rangle\]
in which the second map is induced by the map $\alpha_n$. By the previous theorem, the second map is a weak equivalence, therefore, it is sufficient to prove that the composite map $\beta_n:K_n\to\langle F_n\rangle$ is a weak equivalence. 

This fact follows immediately from the following three claims.
\begin{enumerate}
\item The space $\langle F_n\rangle$ is an Eilenberg-MacLane of type $(\Z,n)$.
\item The class $\iota_n$ is in the image of $H^n(\beta_n)$, where $\iota_n$ denotes the class in $H^n(K_n;\Z)$ corresponding to $1\in\Z$ through the Hurewicz isomorphism~:
\[H^n(K_n;\Z)\cong\Hom(\pi_n(K_n),\Z)\cong\Z\]
\item A map $f:K_n\to X$ where $X$ is an Eilenberg-MacLane space of type $(\Z,n)$ is a weak equivalence if and only if $\iota_n$ is in the image of $H^n(f)$.
\end{enumerate}

The proof of claim (1) is through the free-forgetful adjunction between cosimplicial abelian groups and cosimplicial binomial rings~:
\[\langle F_n\rangle:=\map_{\cbring}(F_n,\Z)\cong\map_{\cat{cAb}}(\Z[-n],\Z)\]
(note that $F_n$ and $\Z[-n]$ are cofibrant in the respective model categories).

Now, we prove claim (3). Clearly the fact that $\iota_n$ is in the image is a necessary condition. Conversely, if $\iota_n$ is in the image, then the map $H^n(f)$ is an isomorphism (indeed a morphism of abelian groups that are abstractly isomorphic to $\Z$ is an isomorphism if and only if it is surjective). By the universal coefficient theorem, we deduce that $H_n(f)$ is an isomorphism. Finally by Hurewicz theorem, we deduce that $\pi_n(f)$ is an isomorphism. Since all the other homotopy groups are trivial, we conclude that $f$ is a weak equivalence.

Finally in order to prove (2), we consider the composite
\[\gamma_n:F_n\to\Z^{\langle F_n\rangle }\to \Z^{K_n}\]
where the first map is the counit of the adjunction and the second map is induced by $\beta_n$. Applying $H^n$ to this map, we see that it is enough to prove that $\iota_n$ is is the image of $H^n(\gamma_n)$. By definition, the map $\beta_n$ is adjoint to the map $\alpha_n:F_n\to \Z^{K_n}$, it follows that the composite $\gamma_n$ is homotopic to the map $\alpha_n$. But by construction the class $\iota_n$ is in the image of $H^n(\alpha_n)$.
\end{proof}

\section{Main theorem}

\begin{defi}\label{defi : convergent tower}
Let $\{X_n\}_{n\in\mathbb{N}}$ be a tower of spaces. We say that this tower is \emph{convergent} if the connectivity of the map $X_{n+1}\to X_n$ diverges to $+\infty$.
\end{defi}

\begin{lemm}\label{lemm : homology of convergent tower}
Let $R$ be a ring. Let $\{X_n\}_{n\in\mathbb{N}}$ be a convergent tower, then the canonical map
\[H_*(\on{holim}_n X_n;R)\to\on{lim}_nH_*(X_n;R)\]
is an isomorphism. The analogous statement holds for cohomology.
\end{lemm}

\begin{proof}
Let us fix a homological degree $i$. For any $n$, the map under consideration fits in a commutative diagram
\[
\xymatrix{
H_i(X_n;R)\ar[d]\ar[dr]& \\
H_i(\on{holim}_nX_n;R)\ar[r]&\lim_nH_i(X_n;R)
}
\]
By the convergence hypothesis, the map $X_n\to \on{holim}_nX_n$ induces an isomorphism in a range of homotopy groups that diverges to +$\infty$ with $n$. Using standard algebraic topology arguments, we deduce that the vertical map is an isomorphism for $n$ large enough. Similarly, the map $H_i(X_{n+1};R)\to H_i(X_n;R)$ is an isomorphism for $n$ large so the diagonal map is an isomorphism for $n$ large. 
\end{proof}

\begin{lemm}\label{lemm : nilpotent spaces}
Let $U\subset\cat{sSet}$ be a full subcategory of the category of simplicial sets satisfying the following conditions
\begin{enumerate}
\item If $X\in U$ and $Y$ is weakly equivalent to $X$, then $Y\in U$.
\item If $X$ and $Y$ are in $U$ so is $X\times Y$.
\item If $X\to Y\to Z$ is a homotopy fiber sequence in which $Z$ is in $U$ and simply connected and $Y$ is in $U$, then $X$ is in $U$. 
\item If $\ldots\to X_n\to X_{n-1}\to \ldots X_0$ is a convergent tower in which all spaces are in $U$ and if the maps $X_n\to X_{n-1}$ are principal fibrations with fiber $K(A,n)$ with $A$ a finitely generated abelian group and $n\geq 1$, then the homotopy limit of the tower is an element of $U$.
\item The spaces $K(\Z,n)$ are in $U$ for any $n\geq 1$.
\end{enumerate}
Then $U$ contains all homotopy types that are nilpotent and of finite type.
\end{lemm}

\begin{proof}
Conditions (2) and (5) imply that the spaces $K(A,n)$ are in $U$ for any finitely generated free abelian group $A$ and any integer $n\geq 1$. Now let $A$ be any finitely generated abelian group and $0\to F_1\to F_2\to A\to 0$ be a short exact sequence with $F_1$ and $F_2$ two finitely generated free abelian groups. Then, we have a fiber sequence
\[K(A,n-1)\to K(F_1,n)\to K(F_2,n)\]
which, by condition (3), implies that $K(A,n)$ is in $U$ for any $n\geq 1$. By \cite[Proposition V.6.1]{goerssjardine}, a space that is nilpotent of finite type can be expressed as the limit of a tower
\[\ldots\to X_{n+1}\to X_n\to X_{n-1}\to \ldots\]
in which, for all $n$, the map $X_n\to X_{n-1}$ fits in a fiber sequence
\[X_n\to X_{n-1}\to K(A_n,k_n)\]
where $A_n$ is a finitely generated abelian group and $k_n\geq 2$ grows to infinity with $n$. Therefore, condition (4) immediately gives us the conclusion.
\end{proof}

\begin{theo}\label{theo : main}
The derived unit map
\[X\to\langle \Z^X\rangle\]
is a weak equivalence if $X$ is a simplicial set which is connected nilpotent and of finite type.
\end{theo}

\begin{proof}
Let us denote by $U$ the full subcategory of $\cat{sSet}$ spanned by the simplicial sets satisfying the conditions of the theorem. By the previous lemma, it suffices to prove that $U$ satisfies the five conditions. Clearly, condition (1) is satisfied. Condition (2) follows from the K\"unneth quasi-isomorphism
\[\Z^{X\times Y}\simeq \Z^X\otimes\Z^Y.\]
and the observation that the tensor product in $\cbring$ is the homotopy coproduct (Proposition \ref{prop : U preserves hocolim}). Condition (3) follows from the convergence of the Eilenberg-Moore spectral sequence. Without loss of generality, we may assume that $Y\to Z$ is a fibration and that $X$ is the actual fiber at a given base point $y\in Y$. Then by \cite[Theorem A.1]{toenprobleme}, the natural map
\[\Z^{Y}\otimes^{\mathbb{L}}_{\Z^Z}\Z\to\Z^X\]
is a quasi-isomorphism of cosimplicial rings and cosimplicial binomial rings (Proposition \ref{prop : U preserves hocolim}). Since, the functor $\langle -\rangle$ sends homotopy colimits to homotopy limits, we obtain the desired result.

We prove (4). Let $\ldots\to X_n\to X_{n-1}\to \ldots X_0$ be a convergent tower in which all spaces are in $U$ (the other conditions are not necessary). Let us denote by $Y$ the homotopy limit of the tower. Then the map
\[H^k(X_n;\Z)\to H^k(Y;\Z)\]
is an isomorphism for $n$ big enough. It follows that the map
\[\mathrm{hocolim}_n\Z^{X_n}\to\Z^Y\]
is a weak equivalence. Again, using that the functor $\langle -\rangle$ sends homotopy colimits to homotopy limits, we obtain the desired result.

Finally claim (5) is exactly Proposition \ref{prop : main theorem for EM spaces}.
\end{proof}

\section{Beyond nilpotent spaces}

In this section, we shall prove the following Theorem

\begin{theo}\label{theo : BK}
Let $X$ be a connected simplicial set of finite type, then the obvious map
\[X\to \langle\Z^X\rangle\]
is weakly equivalent to the Bousfield-Kan $\Z$-completion map $X\to \Z_{\infty}(X)$.
\end{theo}

We shall prove this theorem by using the work of Isaksen. In \cite{isaksencompletions}, for any ring $R$, Isaksen constructs two model structures on the category of pro-simplicial sets. In one of them the weak equivalence are the $R$-cohomology equivalences and  in the other one, they are the $R$-homology equivalences. Recall that given a pro-simplicial set $Y=\{Y_{\alpha}\}_{\alpha\in A}$. Its cohomology is defined by the formula
\[H^*(Y;R):=\on{colim}_{\alpha}H^*(Y_{\alpha};R)\]
and its homology is the pro-$R$-module $\alpha\mapsto H_*(Y_{\alpha};R)$. 

We denote the cohomological (resp. homological) model structure on $\cat{Pro(sSet)}$ by $L^R\cat{Pro(sSet)}$ (resp. $L_R\cat{Pro(sSet)}$). Isaksen observes that a homology equivalence is a cohomology equivalence (by a pro-version of the universal coefficient theorem) but the converse does not hold. The two model structure share the same cofibrations so the cohomological model structure is a localization of the homological one.

\begin{prop}\label{prop : equivalence}
Let $X$ be a connected space of finite type. Let $X\to \{Y_{\alpha}\}$ be a fibrant replacement of $X$ in $L^{\Z}\cat{Pro(sSet)}$. Then the map $X\to \{Y_{\alpha}\}$ is a homology equivalence and so is a fibrant replacement in $L_\Z\cat{Pro(sSet)}$.
\end{prop}

\begin{proof}
We wish to show that the obvious map 
\[\on{colim}_\alpha H^i(Y_\alpha;M)\to H^i(X;M)\]
is an isomorphism for any abelian group $M$. By assumption it is an isomorphism for $M=\Z$. From this we deduce easily that it is an isomorphism for $M$ finitely generated. In order to prove it for a general $M$ it suffices that both the source and the target preserve filtered colimits in the $M$ variable.

For this, we consider the following commutative diagram
\[\xymatrix{
\on{colim}_{\alpha}\on{Ext}(H_{i-1}(Y_{\alpha};\Z),M)\ar[r]\ar[d]&\on{colim}_{\alpha}H^i(Y_{\alpha},M)\ar[r]\ar[d]&\on{colim}_{\alpha}\Hom(H_i(Y_{\alpha};\Z),M)\ar[d]\\
\on{Ext}(H_{i-1}(X;\Z),M)\ar[r]&H^i(X,M)\ar[r]&\Hom(H_i(X;\Z),M)\\
}
\]
in which both rows are the short exact sequences given by the universal coefficient theorem. The map of interest to us is the middle vertical map. The exactness of these rows implies that it is enough to prove that the four external nodes in this diagram preserve filtered colimits in the $M$ variable. By Lemma \ref{lemm : hom and ext preserve filtered limits} below, it suffices to prove that, for each $i$, $H_i(Y_{\alpha};\Z)$ is isomorphic to a pro-finitely generated abelian group. By \cite[Theorem 3.3]{isaksencompletions}, we may assume that $Y_{\alpha}$ is at the top of a finite tower of fibrations with fibers of the form $K(\Z,n)$ with $n\geq 1$. It is well-known that the homology of Eilenberg-MacLane spaces of type $(\Z,n)$ is finitely generated in each degrees. Thus, an easy inductive argument using the Serre spectral sequence shows that this is also the case for the homology of the spaces $Y_{\alpha}$.
\end{proof}

\begin{rem}
It is clear that the proof above holds when $\Z$ is replaced by a principal ideal domain. The case of a field coefficient is treated in \cite[Proposition 3.11]{hoyoisetale}. This Proposition gives a partial answer to \cite[Question 11.2]{isaksencompletions}
\end{rem}

\begin{lemm}\label{lemm : hom and ext preserve filtered limits}
Let $R$ be any ring. Let $\{A_{\alpha}\}$ be a pro-$R$-module group such that $A_{\alpha}$ is finitely generated for each $\alpha$. Then the functor
\[M\mapsto \Hom_{\cat{Pro(Mod}_R)}(\{A_{\alpha}\},M)\]
from $\cat{Mod}_R$ to $\cat{Mod}_R$ preserves filtered colimits. The same statement holds for the functor
\[M\mapsto \on{Ext}_{\cat{Pro(Mod}_R)}(\{A_{\alpha}\},M)\]
\end{lemm}

\begin{proof}
This functor is the filtered colimit of the functors
\[M\mapsto \Hom_{\cat{Mod}_R}(A_{\alpha},M)\]
which all preserve filtered colimits. The case of $\on{Ext}$ is similar since by \cite[Lemma 4.2]{isaksencompletions}, we have an isomorphism
\[\on{Ext}_{\cat{Pro(Mod}_R)}(\{A_{\alpha}\},M)\cong \on{colim}_{\alpha}\on{Ext}_{\cat{Mod}_R}(A_{\alpha},M).\]
\end{proof}

\begin{proof}[Proof of Theorem \ref{theo : BK}]

Let $X\to Y=\{Y_{\alpha}\}_{\alpha\in A}$ be a fibrant replacement of $X$ in $L^\Z\cat{Pro(sSet)}$. By the previous proposition, this is also a fibrant replacement in $L_{\Z}\cat{Pro(sSet)}$. By \cite[Theorem 3.3]{isaksencompletions}, we may assume that each space $Y_\alpha$ is nilpotent in the sense of \cite[Definition 3.1]{isaksencompletions}. Such spaces are not necessarily connected but since $X$ is connected, we may restrict to the connected component $Z_\alpha$ of each $Y_\alpha$ that is hit by the map from $X$. By \cite[Lemma 3.6]{isaksencompletions}, this component is still nilpotent in the sense of \cite[Definition 3.1]{isaksencompletions} and by \cite[Paragraph after question 10.4]{isaksencompletions}, they are nilpotent in the usual sense. Moreover the map $X\to\{Z_\alpha\}$ is clearly still a homology equivalence.

By \cite[Proposition 7.3]{isaksencompletions}, another fibrant replacement of $X$ is given by a strict fibrant replacement of the tower $n\mapsto P_n\Z_n(X)$ with $P_n$ the $n$-th Postnikov section. In other words, the two pro-spaces $\{Z_{\alpha}\}$ and $\{P_n\Z_n(X)\}$ are weakly equivalent in the strict model structure on pro-spaces, so we have a weak equivalence
\[\on{holim}Z_{\alpha}\simeq \on{holim}P_n\Z_n(X)\simeq \Z_{\infty}(X)\]
where the last weak equivalence follows from Milnor's short exact sequence. On the other hand, since $X\to\{Z_{\alpha}\}$ is a $\Z$-cohomology equivalence, we have a weak equivalence of cosimplicial binomial rings $\on{hocolim}\Z^{Z_{\alpha}}\to\Z^X$ and so we obtain a chain of weak equivalences
\[\langle \Z^X\rangle\simeq \langle\on{hocolim}\Z^{Z_{\alpha}}\rangle\simeq \on{holim}Z_{\alpha}\simeq \Z_{\infty}(X).\]
where the second equivalence comes from the fact that the spaces $Z_{\alpha}$ are nilpotent of finite type and Theorem \ref{theo : main}.
\end{proof}

\section{Other localizations}

In this section, for $R$ a commutative ring, we denote by $X\to L_R(X)$ the  localization of a space $X$ with respect to $R$-homology isomorphism. 

\begin{prop}\label{prop : localization}
Let $R$ be a subring of $\Q$ or the ring  $\F_p$ for $p$ a prime number. Let $F:\cat{sSet}\to\cat{sSet}$ be a functor with a natural transformation $\id\to F$ satisfying the following properties.
\begin{enumerate}
\item It preserves weak equivalences.
\item It preserves products up to weak equivalences, i.e. the canonical map
\[F(X\times Y)\to F(X)\times F(Y)\]
is a weak equivalence for all spaces $X$ and $Y$.
\item It sends homotopy fiber sequences with simply connected base to homotopy fiber sequences.
\item It preserves homotopy limit of convergent towers of principal fibration with fibers $K(A,n)$ with $n\geq 1$  and $A$ a finitely generated abelian group.
\item The map $K(\Z,n)\to F(K(\Z,n))$ is an $R$-localization for any $n\geq 1$.
\end{enumerate}
Then, the map $X\to F(X)$ is an $R$-localization for any finite type nilpotent space. 
\end{prop}

\begin{proof}
First, we check that $R$-localization satisfy the five properties above. Properties (1), (2) and (5) are straightforward. Property (3) follows from the fact that $R$-localization preserves simply connected spaces and the Eilenberg-Moore spectral sequence for fiber sequences (which converges when the base is simply connected).  It remains to prove property (4). Let $\{X_n\}$ be a tower satisfying the hypothesis. Because of property (3), the localized tower $\{L_R(X_n)\}$ is still a convergent tower. The canonical map
\[\on{holim}_nX_n\to\on{holim}_nL_R(X_n)\]
induces an isomorphism in homology because of Lemma \ref{lemm : homology of convergent tower} and the target is $R$-local because $R$-local spaces are stable under homotopy limits. It follows that the target must be the $R$-localization of the source as desired.

Now, we prove the proposition. Because of the previous paragraph, the class of spaces for which the map $X\to F(X)$ is an $R$-localization satisfies the five conditions of Proposition \ref{prop : equivalence}. So it must contain all of the nilpotent finite type spaces.
\end{proof}

\begin{prop}
Let $X$ be a finite type connected space. Let $X\to\{Y_\alpha\}$ be a fibrant replacement of $X$ in $L^\Z(\cat{Pro(sSet))}$. Then, the map
\[X\to\on{holim}_{\alpha} L_R(Y_\alpha)\]
is equivalent to Bousfield-Kan $R$-completion.
\end{prop}

\begin{proof}
In order to simplify notations, we write $L$ instead of $L_R$. Arguing as in the Proof of Theorem \ref{theo : BK}, we may assume that each of the spaces $Y_\alpha$ is connected. By Proposition \ref{prop : equivalence}, the map
\[X\to\{Y_\alpha\}\]
is a $\Z$-homology equivalence and thus an $R$-homology equivalence, and the map
\[\{Y_\alpha\}\to \{L(Y_\alpha)\}\]
is an $R$-homology equivalence. By \cite[Propsition 7.3]{isaksencompletions}, we can construct a fibrant replacement of $\{L(Y_\alpha)\}$ in $L_R\cat{Pro(sSet)}$ by taking a strictly fibrant replacement of the pro-space $\{P_nR_nL(Y_\alpha)\}$. We denote by by $\{P_nR_nL(Y_\alpha)\}^f$ this strict fibrant replacement. It follows from the discussion above that the composite
\[X\to \{Y_\alpha\}\to \{L(Y_\alpha)\}\to \{P_nR_nL(Y_\alpha)\}\to \{P_nR_nL(Y_\alpha)\}^f\]
is a fibrant replacement of $X$ in $L_R\cat{Pro(sSet)}$. By \cite[Proposition 7.3]{isaksencompletions}, another fibrant replacement of $X$ is simply given by
\[X\to \{P_nR_nX\}^f\]
In particular, there is a weak equivalence in the strict model structure on pro-spaces
\[\{P_nR_nX\}\simeq \{P_nR_nL(Y_\alpha)\}.\]
It follows that the two pro-spaces have the same homotopy limit so we have
\[R_\infty(X)\simeq \on{holim}_{\alpha,n}P_nR_nL(Y_\alpha).\]

Each of the spaces $L(Y_\alpha)$ is a nilpotent $R$-local space so we have, for each $\alpha$ a weak equivalence
\[\on{holim}_nP_nR_nL(Y_\alpha)\simeq L(Y_\alpha).\]
Putting together these two weak equivalences, we find
\[R_\infty(X)\simeq \on{lim}_\alpha L(Y_\alpha)\]
as desired.
\end{proof}

We are now equipped to prove the main theorem of this section.

\begin{theo}\label{theo : localizations}
Let $R$ be a subring of $\Q$ (resp. the ring $\Z_p$). Let $X$ be a connected space of finite type. Then the canonical map
\[X\to\mathbb{R}\map_{\cbring}(\Z^X,R)\]
is equivalent to Bousfield-Kan $R$-completion (resp. $\F_p$-completion).
\end{theo}

\begin{proof}
Let $X\to\{Y_\alpha\}$ be a fibrant replacement in $L^\Z\cat{Pro(sSet)}$. Then, we have the following sequence of weak equivalences
\[\R\map(\Z^X,R)\simeq \R\map(\on{hocolim}_\alpha \Z^{Y_\alpha},R)\simeq\on{holim}_\alpha\R\map(\Z^{Y_\alpha},R).\]
so according to the previous proposition, it suffices to show that the map 
\[X\to\R\map(\Z^X,R)\]
is an $R$-localization (resp. $\F_p$-localization) for nilpotent finite type spaces. For this, it suffices to show that the functor
\[X\mapsto \R\map(\Z^X,R)\] 
satisfies the five conditions of Proposition \ref{prop : localization}. But this is proved exactly as Theorem \ref{theo : main}.
\end{proof}

\section{Binomial affine homotopy type}\label{section : binomial homotopy type}

We denote by $\on{Spec}$ the functor from $\bring\op$ to $\Fun(\bring,\cat{Set})$ given by
\[\on{Spec}(A)(R)=\Hom_{\bring}(A,R).\]
More generally, let us consider the category $\Fun(\bring,\cat{sSet})$ of covariant functors from $\bring$ to $\cat{sSet}$. For $A$ an object in $\cbring$, we also denote by $\on{Spec}$ the functor from $\cbring\op\to\Fun(\bring,\cat{sSet})$ given by
\[\on{Spec}(A)(R)=\map_{\cbring}(A,R).\]
It is clear that there is no conflict of notations since $\on{Spec}(A)(R)$ is a constant simplicial set when $A$ is a constant cosimplicial binomial ring.
We denote by $\R\on{Spec}(A)$ the functor
\[R\mapsto \map(A^c,R)\]
for $A^c$ a cofibrant replacement of $A$. 

\begin{prop}\label{prop : ff 1}
The functor $A\mapsto \R\on{Spec}(A)$ from $c\mathcal{BR}ing\op$ to $\Fun(\bring,\mathcal{S})$ is homotopically fully faithful.
\end{prop}

\begin{proof}
This is very similar to \cite[Corollaire 2.2.3]{toenchamps}.  We define a \emph{binomial ring structure} on an object $A$ of a category $\cat{C}$ with finite products to be the data of a lift of the functor
\[B\mapsto \Hom_{\cat{C}}(B,A)\]
from $\Fun(\cat{C}\op,\cat{Set})$ to $\Fun(\cat{C}\op,\bring)$. 

Tautologically, the functor
\[A\mapsto \Hom_{\bring}(\on{Num}[x],A)\]
can be lifted from sets to binomial rings. This means that the binomial ring $\on{Num}[x]$ has a binomial ring structure in the category $\bring\op$. The functor $\on{Spec}$ being a product preserving functor from $\bring\op$ to $\Fun(\cbring,\cat{sSet})$, we deduce that $\tG:=\on{Spec}(\on{Num}[x])$ has a canonical binomial ring structure in $\Fun(\bring,\cat{Set})$

This implies that the functor
\[\mathcal{O}:X\mapsto ([n]\mapsto \Hom_{\Fun(\bring,\cat{Set})}(X_n,\tG))\]
lands in $\cbring$ and clearly, this functor is left adjoint to $\on{Spec}$. Now, we observe that the adjunction
\[\mathcal{O}:\Fun(\bring,\cat{sSet})\leftrightarrows \cbring\op:\on{Spec}\]
is a Quillen adjunction when we give $\Fun(\bring,\cat{sSet})$ the projective model structure. So we wish to prove that the unit map
\[A\to \mathbb{L}\mathcal{O}\on{Spec}(A)\]
is a weak equivalence for any cofibrant $A$. For this, we first observe that the canonical map
\[A\to \on{holim}_{[n]\in\Delta}A^n\]
is a weak equivalence. It follows that $\on{Spec}(A)\simeq \on{hocolim}\R\on{Spec}(A_n)$. Now, since $\on{Spec}(A_n)$ is representable, it is in particular cofibrant. It follows that 
\[\mathbb{L}\mathcal{O}\on{Spec}(A)\simeq \on{holim}_{[n]\in\Delta} \mathcal{O}\on{Spec} (A_n)\]
Thus, in order to conclude the proof, it suffices to observe that the map
\[R\to \mathcal{O}\on{Spec}(R).\]
is an isomorphism for $R$ an object of $\bring$. This is simply Yoneda's lemma.
\end{proof}

\begin{rem}\label{rem : tilde}
The object $\tG\in\Fun(\bring,\cat{Grp})$ is simply the binomial version of the additive group~:
\[R\mapsto (R,+).\]
We are reluctant to denote this object by $\mathbb{G}_a$ as the representing object is not the same as the representing object of $\mathbb{G}_a:\cat{Ring}\to\cat{Grp}$. 

More precisely, we may think of binomial rings as affine geometric objects over a deeper base than $\on{Spec}(\Z)$, that we denote by $\on{Spec}(\Z^{bin})$. Then the forgetful functor
\[U:\bring\to\ring\]
should be thought of as the base change $R\mapsto \Z\otimes_{\Z^{bin}}R$. Note that this base change notation does not have any mathematical content, it is simply a rewriting of the forgetful functor. In any case, the object $\mathbb{G}_a$ is not the base change of $\tG$ from $\on{Spec}(\Z^{bin})$ to $\on{Spec}(\Z)$. The latter base change is what To\"en calls the Hilbert additive group in \cite{toenprobleme}.
\end{rem}

We can go one step further and extend $X^{bin}$ to all commutative rings. We denote by $U_!$ the left Kan extension functor along the forgetful functor~:
\[U_!:\Fun(\bring,\cat{sSet})\to\Fun(\ring,\cat{sSet})\]

\begin{prop}\label{prop : ff2}
The functor $U_!$ preserves weak equivalences and the induced functor
\[U_!:\Fun(\bring,\mathcal{S})\to\Fun(\ring,\mathcal{S})\]
is fully faithful.
\end{prop}

\begin{proof}
By standard categorical nonsense, this left Kan extension functor is simply given by precomposition with $\bin_U$ the right adjoint functor to $U$ so it obviously preserves weak equivalences. It follows that the adjunction
\[U_!:\Fun(\bring,\cat{sSet})\leftrightarrows \Fun(\ring,\cat{sSet}):U^*\]
is a Quillen adjunction in which both functors preserve weak equivalences. Thus, in order to prove that the derived left adjoint is fully faithful, it suffices to prove that $U_!$ is fully faithful at the $1$-categorical level. This follows immediatly from the fact that $U$ itself is fully faithful.
\end{proof}

\begin{nota}
For $X$ a simplicial set, we write $X^{bin}$ for $\R\on{Spec}(\Z^X)\in\Fun(\bring,\cat{sSet})$. The assignment $X\mapsto X^{bin}$ descends to a functor
\[\mathcal{S}\to\Fun(\bring,\mathcal{S})\]
By Theorem \ref{theo : main} and Proposition \ref{prop : ff 1}, this functor is fully faithful when restricted to spaces that are nilpotent and of finite type. We also denote by $X^{bin}$ the functor $U_!X^{bin}\in\Fun(\ring,\cat{sSet})$. This is not a conflict of notation since, for $R$ a binomial ring, we have $U_!X^{bin}(R)\cong X^{bin}(R)$. Moreover, by Proposition \ref{prop : ff2}, the assignment $X\mapsto X^{bin}$ from $\mathcal{S}$ to $\Fun(\ring,\mathcal{S})$ is also fully faithful when restricted to spaces that are nilpotent and of finite type.
\end{nota}

\begin{rem}
The functor $X^{bin}:\bring\to\cat{sSet}$ is represented by the cosimplicial algebra $\Z^X$. That is, it is given by 
\[R\mapsto \map_{\cbring}(A,R)\]
for $A\to\Z^X$ a cofibrant replacement in $\cbring$. Hence, following To\"en \cite{toenchamps}, we may call it a ``binomial affine stack''. 

The extension of $X^{bin}$ to $\ring$ is given by
\[R\mapsto \map_{\cbring}(A,\bin_U(R))\cong\map_{\cring}(UA,R)\]
but it is almost never an affine stack in the sense of To\"en. This is because the algebra $UA$ is not cofibrant in $\cring$ in general. In fact, if we replace $UA$ be a cofibrant algebra, the resulting functor will be precisely what To\"en calls the ``affinisation'' of the homotopy type $X$ which is denoted by $(X\otimes\Z)^{uni}$ in \cite{toenprobleme}. By \cite[Corollaire 5.1]{toenprobleme}, the only space that is simply connected and of finite type for which $X^{bin}$ coincides with $(X\otimes\Z)^{uni}$ is the point.
\end{rem}

\begin{prop}\label{prop : points of binomial homotopy type}
Let $X$ be a connected space of finite type.
\begin{enumerate}
\item If $R$ is a subring of $\Q$, then the obvious map $X\to X^{bin}(R)$ coincides with Bousfield-Kan $R$-completion.
\item The obvious map $X\to X^{bin}(\Z_p)$ coincides with Bousfield-Kan $p$-completion.
\item If $k$ is a finite field of characteristic $p$, then the map $X\to X^{bin}(k)$ coincides with Bousfield-Kan $p$-completion.
\end{enumerate}

\end{prop}

\begin{proof}
The first two claims are the content of Theorem \ref{theo : localizations}

The proof of (3) simply comes from the fact that for any finite field of characteristic $p$ we have $\bin_U(k)=\Z_p$.
\end{proof}

\section{A binomial Grothendieck-Teichmüller group}

We denote by $\cat{BStack}$ the category $\Fun(\bring,\cat{sSet})$. This category is enriched over itself by the formula
\[\ehom(X,Y)(R)=\map_{\Fun(\bring,\cat{sSet})}(X\times\on{Spec}(R),Y).\]
If $X=Y$ is cofibrant and fibrant, then the object $\End(X)=\ehom(X,X)$ is a functor from $\bring$ to fibrant simplicial monoids. We may apply the homotopy unit functor objectwise and we obtain a functor from $\bring$ to grouplike simplicial monoids that we call $\Aut(X)$. More generally, if $X$ is not cofibrant-fibrant, we denote by $\R\Aut(X)$ the result of applying $\Aut$ to a cofibrant-fibrant replacement of $X$.

\begin{exam}\label{exam : automorphisms of torus}
Recall that $\tG\in\Fun(\bring,\cat{Grp})$ denotes the binomial additive group :
\[R\mapsto (R,+)\] 
Take $X=T^n$ the $n$-dimensional torus. Then, we have $X^{bin}\simeq \on{Spec}(\sym^{bin}(\Z^n[-1]))$. so we have $X^{bin}=K(\tG^n,1)$. For $R$ a binomial ring, we have the following sequence of weak equivalences
\[\pi_0\R\Aut(X^{bin})(R)\cong \pi_0\map_{\cbring}(\sym^{bin}(\Z^n[-1]),\sym^{bin}(\Z^n[-1])\otimes R)^{\times}\cong \mathrm{GL}_n(R) \]
\end{exam}

This construction extends to diagrams $I\to\cat{BStack}$. The category of such diagram is enriched over $\cat{BStack}$ by the following equalizer diagram
\[\ehom(F,G)\to\prod_{i\in I}\ehom(F(i),G(i))\rightrightarrows\prod_{f:i\to j}\ehom(F(i),G(j))\]
In particular, if $F:I\to\cat{BStack}$ is a functor, we may form $\R\Aut(F)$ which is a functor from $\bring$ to grouplike simplicial monoids.

\begin{prop}\label{prop : R-points of RAut}
Let $F:I\to\cat{sSet}$ be a functor taking values in spaces that are connected of finite type. Let $F^{bin}:I\to\cat{BStack}$ the functor obtained by applying the construction $X\mapsto X^{bin}$ objectwise. Then
\begin{enumerate}
\item For $R$ a subring of $\Q$, we have a weak equivalence of simplicial monoids
\[\R\Aut(F^{bin})(R)\simeq \R\Aut(R_\infty(F))\]
\item For $p$ a prime, we have a weak equivalence of simplicial monoids
\[\R\Aut(F^{bin})(\Z_p)\simeq \R\Aut((\F_{p})_{\infty}(F))\]
\end{enumerate}
\end{prop}

\begin{proof}
This is an immediate consequence of Proposition \ref{prop : points of binomial homotopy type}.
\end{proof}

\begin{cons}
Let $D_2$ be the little $2$-disks operad. We view $D_2$ as a dendroidal space, i.e. a diagram $\Omega\op\to\cat{sSet}$. Alternatively, we may view $D_2$ as a weak operad in the sense of \cite[Definition 2.8]{horelprofinite} (i.e. a functor from $\Psi\op$ to $\cat{sSet}$ satisfying a Segal condition for $\Psi$ the algebraic theory of single-colored operads). For $R$ a binomial ring, we define
\[\mathrm{GT}(R)=\pi_0(\R\Aut_{\Fun(\Omega\op,\cat{BStack})}(D_2^{bin})(R))\]
This is our version of the Grothendieck-Teichmüller group. Evaluation in arity $2$ induces a natural transformation
\[\R\Aut_{\Fun(\Omega\op,\cat{BStack})}(D_2^{bin})\to\R\Aut_{\cat{BStack}}((S^1)^{bin}).\]
Taking $\pi_0$ and using Example \ref{exam : automorphisms of torus}, this induces a natural transformation of functors from $\bring$ to groups~:
\[\chi : \mathrm{GT}\to\tilde{\mathbb{G}}_m\] 
where $\tilde{\mathbb{G}}_m(R)$ is the binomial multiplicative group given by $R\mapsto (R^{\times},.)$ (see Remark \ref{rem : tilde} for the notation).
\end{cons}

\begin{theo}
The group $\mathrm{GT}$ is related to the pro-$p$ Grothendieck-Teichmüller group $\hat{\mathrm{GT}}_p$ and rational Grothendieck-Teichmüller group $\hat{\mathrm{GT}}_{\Q}$ as follows.
\begin{enumerate}
\item We have $\mathrm{GT}(\Z_p)\cong\hat{\mathrm{GT}}_p$.
\item We have $\mathrm{GT}(\Q)\cong \hat{\mathrm{GT}}_{\Q}$.
\end{enumerate}
Moreover, in case (1) the map $\chi$ induces the pro-$p$ cyclotomic character $\hat{\on{GT}}_p\to\Z_p^\times$ and in case (2) the map $\chi$ induces the rational cyclotomic character $\hat{\on{GT}}_\Q\to\Q^\times$.
\end{theo}

\begin{proof}
By Proposition \ref{prop : R-points of RAut}, we find that
\[\on{GT}(\Z_p)\simeq \pi_0\R\Aut((\F_p)_\infty D_2)\]
which according to \cite[Theorem 5.7]{horelautomorphisms} is indeed isomorphic to $\hat{\on{GT}}_p$. Similarly, we have an isomorphism
\[\on{GT}(\Q)\cong \hat{\on{GT}}_\Q\]
by the main theorem of \cite{fressehomotopy}.
\end{proof}

\begin{rem}
We could do a similar construction with the framed little $2$-disks operad. The resulting group will have the same $\Z_p$ points by the main result of \cite{boavidaoperads} (see also \cite[Remark 5.8]{horelautomorphisms}). The $\Q$-points of this group will also be isomorphic to $\hat{\on{GT}}_{\Q}$ although we do not know of a reference for this fact.
\end{rem}

\appendix

\section{Conservativity of the bar construction}\label{section : conservativity bar}

Let $A$ be an augmented differential graded algebra over a commutative ring $R$. We assume that $A$ is flat over $R$. We can form the derived tensor product $R\otimes^{\mathbb{L}}_AR$. This can be explicitly modelled by the bar complex
\[B(A)=T^{c}(s\overline{A},d_A+d')\]
where $\overline{A}$ denotes the augmentation ideal of $A$, $T^c$ denotes the cofree conilpotent coalgebra comonad, $d_A$ is induced by the differential of $A$ and $d'$ is the bar differential induced by the algebra structure. The goal of this section is to study the conservativity of the bar construction functor. We are grateful to Dan Petersen for suggesting the proof of Theorem \ref{conservativity of Bar}.

\begin{prop}
\begin{enumerate}
\item Let $C$ be a chain complex of abelian groups, assume that for all prime $p$, the chain complex $C\otimes_{\Z}^{\mathbb{L}}\F_p$ is acyclic and that $C\otimes_{\Z}^{\mathbb{L}}\Q$ is acyclic. Then $C$ is acyclic.
\item Let $f:C\to D$ be a map of chain complexes of abelian groups. Assume that for all prime $p$, the  map $f\otimes_{\Z}^{\mathbb{L}}\F_p$ is a quasi-isomorphism and the map $f\otimes_{\Z}^{\mathbb{L}}\Q$ is a quasi-isomorphism, then $f$ is a quasi-isomorphism.
\end{enumerate}

\end{prop}

\begin{proof}
In this proof, we drop the $\mathbb{L}$ superscript from the notation, all tensor products are derived. First, we observe that part (2) of the proposition follows immediately from part (1). Indeed, a map of chain complexes is a quasi-isomorphism if and only if its cofiber is acyclic and taking cofiber commutes with derived tensor product.

We now prove part (1). First, using the short exact sequences
\[0\to\Z/p^n\to\Z/p^{n+1}\to\Z/p\to 0\]
and an obvious inductive argument, we can conclude, that $C\otimes\Z/p^n$ is acyclic for all $n$. It follows that $C\otimes\Z/N$ is acyclic for all integer $N$. Since the group $\Q/\Z$ is a filtered colimit of finite cyclic groups, we deduce that $C\otimes \Q/\Z$ is also acyclic. Finally, the short exact sequence
\[0\to\Z\to\Q\to\Q/\Z\to 0\]
lets us conclude that $C=C\otimes\Z$ is also acyclic.
\end{proof}

We switch from augmented (co)algebras to non-unital (co)algebras (recall that the functor that sends an augmented algebra to its augmentation ideal is an equivalence of categories and similarly for coalgebras). In the rest of this proof, algebra and coalgebras are non-unital and non-counital.

The bar construction functor lands in coalgebras~:
\[B:\cat{dgAlg}\to\cat{dgCoalg}\]
It has a left adjoint called the cobar construction~:
\[\Omega : \cat{dgCoalg}\to\cat{dgAlg}\]

\begin{prop}
The cobar construction sends quasi-isomorphisms between simply coconnected coalgebras (i.e. such that $C_i=0$ for $i\geq -1$) to quasi-isomorphisms.
\end{prop}

\begin{proof}
The cobar construction is given by 
\[\Omega(C)= (T(s^{-1}C),d_C+d')\]
where $d'$ is constructed using the coalgebra structure. There is an obvious decreasing filtration on $\Omega(C)$ given by length of tensors
\[F_n(C)=\oplus_{i\geq  n}(s^{-1}C)^{\otimes n}\subset T(s^{-1}C)\]
We claim that if $C$ is simply coconnected, the map
\begin{equation}\label{completeness}
\Omega(C)\to \mathrm{lim}_n\Omega(C)/F_n(\Omega(C))
\end{equation}
is a quasi-isomorphism. For this it suffices to observe that
\begin{enumerate}
\item This limit is in fact a homotopy limit.
\item The obvious map
\[H_i(\mathrm{holim}_n\Omega(C)/F_n(\Omega(C)))\to\mathrm{lim}_nH_i(\Omega(C)/F_n(\Omega(C)))\]
is an isomorphism.
\item For each $i$, the composed map
\[H_i(\Omega(C))\to \mathrm{lim}_nH_i(\Omega(C)/F_n(\Omega(C)))\]
is an isomorphism.
\end{enumerate}

The first claim follows from the fact that the transition maps are all epimorphisms in this tower. The second claim uses the Milnor short exact sequence and the fact that the $\mathrm{lim}^1$-term vanishes by the Mittag-Leffler criterion. For the last claim, we observe that $F_n(\Omega(C))$ is concentrated in homological degrees less than $-n$ by the coconnectedness assumption. It follows that for any $i$, the map
\[H_i(\Omega(C))\to H_i(\Omega(C)/F_n(\Omega(C)))\]
is an isomorphism for $n$ large enough.

Now that we know that the map \ref{completeness} is a quasi-isomorphism, the proposition is easy to prove. Let $f:C\to D$ be a quasi-isomorphism between simply-coconnected coalgebras. It suffices to consider the following commutative square
\[\xymatrix{
\Omega(C)\ar[r]\ar[d]_{\Omega(f)}& \mathrm{lim}_n\Omega(C)/F_n(\Omega(C))\ar[d]\\
\Omega(D)\ar[r]& \mathrm{lim}_n\Omega(C)/F_n(\Omega(D))
}
\]
in which both vertical maps are induced by $f$. All we have to do is prove that the right vertical map is a quasi-isomorphism. Since we know that both limits are in fact homotopy limits, it suffices to prove that for each $n$ the map
\[\Omega(C)/F_n(\Omega(C))\to \Omega(D)/F_n(\Omega(D))\]
is a quasi-isomorphism. By an inductive argument it is enough to show that for all $n$, the map induced by $f$~:
\[F_{n-1}(\Omega(C))/F_n(\Omega(C))\to F_{n-1}(\Omega(C))/F_n(\Omega(C))\]
is a quasi-isomorphism which follows immediately from the fact that $f$ is a quasi-isomorphism.
\end{proof}

\begin{theo}\label{conservativity of Bar}
Let $\mathbb{K}$ denote the ring $\Z$ or a field.
Let $f:A\to A'$ be a map of augmented $0$-coconnected dg-algebras over $\mathbb{K}$. Assume that
\[B(f):B(A)\to B(A')\]
is a quasi-isomorphism. Then $f$ is a quasi-isomorphism.
\end{theo}

\begin{proof}
Thanks to the first proposition, we may reduce the case of $\Z$ to the case of a field so from now on, we assume that $\mathbb{K}$ is a field.

Thanks to the second proposition, we know that $\Omega (B(f)):\Omega B(A)\to \Omega B(A')$ is a quasi-isomorphism. Finally, the canonical maps $\Omega B(A)\to A$ and $\Omega B(A')\to A'$ are quasi-isomorphisms by \cite[Corollary 2.3.4]{lodayvallette} which concludes the proof.
\end{proof}

\bibliographystyle{abbrv}
\bibliography{bib}

\end{document}